\documentclass[11pt, one side, article]{memoir}
\usepackage{coll,amsmath,amssymb,amsthm, verbatim, tikz-cd, graphicx, quiver, mathtools, xparse, pgf,pgfplots, color}
\usepackage[utf8]{inputenc}
\usepackage{graphicx, color, hyperref}
\usepackage{dcolumn}
\usepackage{bm}
\usepackage{amsmath}
\usepackage{amsthm}
\usepackage{xparse} 
\usepackage{mathtools} 
\usepackage[caption=false]{subfig}
\usepackage{blkarray} 
\usepackage{amscd,verbatim}
\usepackage{amssymb}
\usepackage[all]{xy}
\usepackage{tikz}
\usetikzlibrary{calc,fit}
\usetikzlibrary{fit}
\usepackage{pgf,pgfplots}
\usepackage{enumitem} 

\newcommand{\define}[1]{\textbf{#1}}
\newcommand{\R}{\mathbb{R}}

\definecolor{mydarkred}{RGB}{233,20,35}
\definecolor{mypurple}{RGB}{120, 35, 160}
\definecolor{mydarkpurple}{RGB}{128, 100, 162}
\definecolor{mybrown}{RGB}{255, 195, 0}
\definecolor{myaqua}{RGB}{29, 153, 168}
\definecolor{myblue}{RGB}{91, 129, 184}  
\definecolor{mygreen}{RGB}{155, 187, 89}  

\definecolor{mybrightblue}{RGB}{0, 140, 255}  

\tikzstyle{species}=[
  circle,draw=black!100, thick, inner sep=0pt,minimum size=6mm
  ] 
\tikzstyle{reaction}=[rectangle,draw=black!100,fill=black!15,thick, inner sep=0pt,minimum size=6mm]

\title{Categorical Perspectives on Chemical Reaction Networks}
\author{Justin Curry, Mauricio Montes}
\date{}
\begin{document}

\maketitle

\begin{abstract}
    We show that the Schur-complement reduction of \cite{Hirono2021} a chemical reaction network (CRN) is the categorical complement of the stoichiometric arrow in the arrow category $[\mathbf{A}_2,\mathbf{Vect}]$. This identifies the ambient category in which topological reduction of chemical reaction networks is functorial and explains the reduced stoichiometric matrix as a universal diagrammatic construction. We further define a reconstruction functor from a restricted subcategory of $[\mathbf{A}_2, \mathbf{Vect}]$ back to CRNs and prove an adjunction with the stoichiometric functor. 
\end{abstract}

\chapter{Introduction}
\label{sec:intro}

Chemical reaction networks (CRNs) serve as a foundational framework for understanding complex systems across chemistry, biology, and engineering. Designing these CRNs requires the construction of detailed models with specified kinetics, initial conditions, and parameter values. Although such models can yield precise predictions, their complexity makes them unwieldy, particularly in biological systems where numerous variables interact in nonlinear ways \cite{Konieczny2023, Tanner2018-hj, FEINBERG1974775, FEINBERG198059}. As a result, there has been lots of research done to study and simplify these CRNs \cite{hycosbm, jost2018hypergraph, PhysRevX.13.021040, MorphismsReaction, FEINBERG19891819}.

Existing reduction methods, such as timescale separation \cite{10.1214/12-AAP841}, lumping \cite{Wei1969, Kuo1969}, sensitivity analysis, optimization techniques, and topological reductions \cite{Hirono2021, rao2012graphtheoreticalapproachanalysismodel}, offer powerful tools for simplifying CRNs. These reductions are invaluable—they not only decrease the number of variables and parameters needed for analysis but also illuminate the key features driving phenomena of interest. However, many of these approaches often require detailed, system-specific information, such as the classification of reactions into fast and slow processes or the parameter dependence of system dynamics. In contrast, topological reductions stand out for their broad applicability, relying solely on the network structure of the CRN, and are independent of specific type of kinetics or parameter values. This independence is particularly significant, as the kinetics of reactions or the values of parameters are often difficult to determine in systems, and may require sophisticated tools to achieve an accurate approximation \cite{doi:10.1021/acs.jpcb.3c03649}.

In this article, we follow a direct approach for formalizing the topological reduction work of \cite{Hirono2021} categorically.
Our definitions are not boiler plate reproductions of the work of either \cite{spivak2015steady} or \cite{Baez_2017}, which are the two works closest to this one.
In particular, by providing a new \emph{pointed} definition of a CRN we are able to describe structural reductions using morphisms, which are subtly different from the morphisms between Petri nets used in \cite{baez2021nets}.
En route to this result, we show that oriented hypergraph homology in the sense of \cite{diestel2020homological}, defines a functor from our category of CRNs.
This proves that the ad hoc object-wise construction of homology used by \cite{Hirono2021} can be promoted to a functor.
We finish by connecting the categorical study of the Schur complement, pioneered in \cite{henselman2017matroids}, to the structural reduction of a CRN and prove that this can be viewed as a diagrammatic complement in the category of arrows on CRNs.

\chapter{The Category of CRNs}

Chemical Reaction Networks (CRNs) are a popular object of study within the applied category theory community.
We now introduce the category we choose to work with, which is slightly different from the presentation in both \cite{Baez_2017} and \cite{Hirono2021}.

\begin{definition}[Chemical Reaction Network, cf. Def. 8 of \cite{Hirono2021}]\label{defn:CRN}
    A (generalized) \define{chemical reaction network (CRN)} is a quadruple $\Gamma = (V,E,s,t)$, consisting of
    \begin{itemize}
        \item a set $V$ of vertices, interpreted as chemical species,
        \item a necessarily non-empty \emph{pointed} set $E_*=E \sqcup \{e_*\}$ of directed hyperedges, interpreted as chemical reactions, where
        \item the two maps $s,t : E_{\ast} \to \mathbb{R}^V = \{f|f:V \to \mathbb{R} \}$, specify the required inputs and resulting outputs of each reaction $e_A\in E_*$.
    \end{itemize}
    Every CRN has an associated \define{stoichiometry matrix} $S_\Gamma=(S_{iA})$, whose entries are
    \[
    S_{iA}:=t(e_A)(v_i)-s(e_A)(v_i),
    \]
    where the $i^{th}$ row corresponds to vertex $v_i$ and the $A^{th}$ column corresponds to edge $e_A$.
    If an edge $e\in E_*$ satisfies $s(e)=t(e)\in \R_{\geq 0}^V$, then we say that $e$ is \define{degenerate}.
    We assume further that the ``do nothing'' edge $e_*\in E_*$ is degenerate and 0, i.e., $s(e_*)=t(e_*)=0\in \R_{\geq 0}^V$.
    This implies that the ``do nothing'' edge is equivalently viewed as a column of zeros that we can drop as needed.
\end{definition}

\begin{example}
    The \define{trivial} (or empty) CRN, written $\varnothing$, is the unique CRN determined by $V=E=\varnothing$.
    Note that in this setting $E_*=\{e_*\}$ admits a natural map to $\R^{\varnothing}=0$, the zero vector space.
\end{example}

\begin{definition}[CRN morphism]\label{defn:CRN-morphism}
    A $\define{CRN morphism},$ $\phi$ from $\Gamma = (V, E, s, t)$ to $\Gamma' = (V', E', s', t')$ is a pair of maps, $(\phi_0, \phi_1)$ where
    \begin{equation}
        \phi_0: \mathbb{R}^{V} \to \mathbb{R}^{V'} \text{ is a linear map, and }
    \end{equation}
    \begin{equation}
        \phi_1: E_* \to E'_* \text{ is a pointed set map,}
    \end{equation}
    which must make the following diagrams commute:
    \begin{equation}
        \begin{tikzcd}
    	E_* & {\mathbb{R}^V} & E_* & {\mathbb{R}^V} \\
    	{E'_*} & {\mathbb{R}^V} & {E'_*} & {\mathbb{R}^V}
    	\arrow["s", from=1-1, to=1-2]
    	\arrow["{\phi_1}", from=1-1, to=2-1]
    	\arrow["{\phi_0}", from=1-2, to=2-2]
    	\arrow["t", from=1-3, to=1-4]
    	\arrow["{\phi_1}", from=1-3, to=2-3]
    	\arrow["{\phi_0}", from=1-4, to=2-4]
    	\arrow["{s'}", from=2-1, to=2-2]
    	\arrow["{t'}", from=2-3, to=2-4]
        \end{tikzcd}
    \end{equation}
\end{definition}

\begin{proposition}\label{prop:CRN-pointed-category}
    Every CRN $\Gamma$ admits an identity morphism to itself.
    Composition of CRN morphisms exists and is associative, thus making the collection of all CRNs, written \textbf{CRN}, into a category.
    Moreover, every $\Gamma\in \mathbf{ob}(\textbf{CRN})$ has a unique map \emph{to} and a unique map \emph{from} the trivial CRN $\varnothing$, thereby making \textbf{CRN} a pointed category. 
\end{proposition}
\begin{proof}
    The existence of identities is clear.
    If $\phi=(\phi_0,\phi_1):\Gamma \to \Gamma'$ and $\phi'=(\phi_0',\phi_1'):\Gamma' \to \Gamma''$ are two CRN morphisms, then the composition $\phi' \circ \phi=(\phi_0'\circ \phi_0,\phi_1'\circ\phi_1)$ exists because composition exists in the category of pointed sets and linear maps. The fact that the composition commutes with $(s,t),(s'',t'')$ is clear from the fact that pasting commutative squares produces larger commuting rectangles.

    \[\begin{tikzcd}
	E_* & {\mathbb{R}^V} & E_* & {\mathbb{R}^V} \\
	{E_*'} & {\mathbb{R}^{V'}} & {E_*'} & {\mathbb{R}^{V'}} \\
	{E_*''} & {\mathbb{R}^{V''}} & {E_*''} & {\mathbb{R}^{V''}}
	\arrow["s", from=1-1, to=1-2]
	\arrow["{\phi_1}"', from=1-1, to=2-1]
	\arrow["{\phi_0}", from=1-2, to=2-2]
	\arrow["t", from=1-3, to=1-4]
	\arrow["{\phi_1}"', from=1-3, to=2-3]
	\arrow["{\phi_0}", from=1-4, to=2-4]
	\arrow["{s'}"', from=2-1, to=2-2]
	\arrow["{\phi'_1}"', from=2-1, to=3-1]
	\arrow["{\phi'_0}", from=2-2, to=3-2]
	\arrow["{t'}"', from=2-3, to=2-4]
	\arrow["{\phi'_1}"', from=2-3, to=3-3]
	\arrow["{\phi'_0}", from=2-4, to=3-4]
	\arrow["{s''}"', from=3-1, to=3-2]
	\arrow["{t''}"', from=3-3, to=3-4]
\end{tikzcd}\]
\end{proof}

\subsection{Categorical Comparisons}

On first glance, a CRN $\Gamma$ appears to be just a special type of functor $G$ from the digon category $\mathbf{D}$, drawn below, to the category of sets $\mathbf{Set}$.

\begin{center}
\begin{tikzpicture}
  \node at (-1,0) {$\mathbf{D}$};

  \node (L1) at (0,0) {$\bullet$}; 
  \node (R1) at (2,0) {$\bullet$}; 

  \draw[->] (L1) to[bend left=30] (R1);
  \draw[->] (L1) to[bend right=30] (R1);

  \draw[->, decorate, decoration={snake, amplitude=.4mm, segment length=2mm, post length=1mm}] 
    (2.5,0) -- node[above, yshift=1mm] {$G$} (3.5,0);

  \node (E) at (4,0) {$E$}; 
  \node (V) at (6,0) {$V$}; 

  \draw[->] (E) to[bend left=30] node[above] {$s$} (V);
  \draw[->] (E) to[bend right=30] node[below] {$t$} (V);

  \node at (7.5,0) {$\mathbf{Set}$};
\end{tikzpicture}
\end{center}

We now explain why this is not the case.
Any functor $G \in \mathbf{ob}([\mathbf{D},\mathbf{Set}])$ is equivalently a \define{directed graph}.
As \cite{Hirono2021} points out, directed graphs only model \emph{monomolecular reactions}, i.e.~reactions with at most one chemical input and one chemical output.
However, non-monomolecular reactions are commonplace. Consider the following example:

\begin{example}[cf. Example 1 from \cite{Hirono2021}]\label{ex:1}
Consider the CRN
\[
\Gamma = (\{v_1, v_2, v_3, v_4, v_5\}, \{e_1, e_2, e_3, e_4, e_5, e_6\})
\]
given by the following set of chemical reactions:
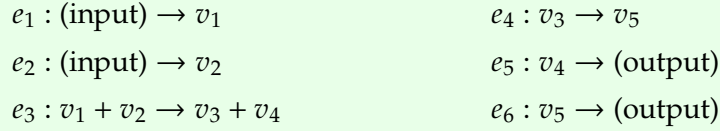
\begin{align*}
e_1&: \text{(input)} \to v_1 & e_4&: v_3 \to v_5 \\
e_2&: \text{(input)} \to v_2 & e_5&: v_4 \to \text{(output)} \\
e_3&: v_1 + v_2 \to v_3 + v_4 & e_6&: v_5 \to \text{(output)}
\end{align*}
This can be drawn as:
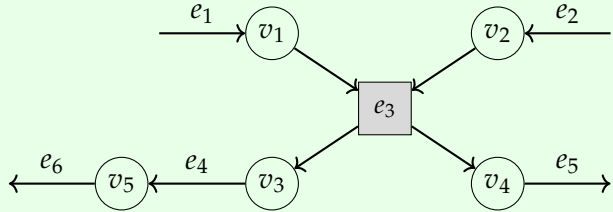
\begin{center}
\begin{tikzpicture}[
    vnode/.style={circle, draw, minimum size=7mm, inner sep=0pt},
    enode/.style={rectangle, draw, fill=gray!30, minimum size=7mm, inner sep=2pt, font=\small}
]
    \node[vnode] (v1) at (0,2) {$v_1$};
    \node[vnode] (v2) at (3,2) {$v_2$};
    \node[vnode] (v3) at (0,0) {$v_3$};
    \node[vnode] (v4) at (3,0) {$v_4$};
    \node[vnode] (v5) at (-2,0) {$v_5$};

    \node[enode] (e3) at (1.5,1) {$e_3$};

    \draw[->, thick] (-1.5, 2) -- node[above] {$e_1$} (v1);
    \draw[->, thick] (4.5, 2) -- node[above] {$e_2$} (v2);
    \draw[->, thick] (v3) -- node[above] {$e_4$} (v5);
    \draw[->, thick] (v4) -- node[above] {$e_5$} (4.5,0);
    \draw[->, thick] (v5) -- node[above] {$e_6$} (-3.5,0);

    \draw[->, thick] (v1) -- (e3);
    \draw[->, thick] (v2) -- (e3);
    \draw[->, thick] (e3) -- (v3);
    \draw[->, thick] (e3) -- (v4);

\end{tikzpicture}
\end{center}
The stoichiometry matrix $S_{\Gamma}$ is
\[
S = \begin{pmatrix}
1 & 0 & -1 & 0 & 0 & 0 \\
0 & 1 & -1 & 0 & 0 & 0 \\
0 & 0 & 1 & -1 & 0 & 0 \\
0 & 0 & 1 & 0 & -1 & 0 \\
0 & 0 & 0 & 1 & 0 & -1
\end{pmatrix}.
\]
\end{example}

\cite{Baez_2017} and others use the grey box representation of reactions to convert the hyperedges of a CRN into vertices of a particular type. 
This then treats CRNs as particular examples of directed bipartite graphs.
We argue that this perspective lacks the clear connection to homology presented in the next section.

\chapter{Quiver Representations, Stoichiometry and Homology}

We now illustrate how to convert a CRN into a type--$A_2$ quiver representation, which is equivalently a matrix.
Simply stated, we can map every CRN $\Gamma$ to its associated stoichiometry matrix $S_{\Gamma}$ and this association is functorial.
Since kernels and cokernels are functorial, and we can interpret these as $H_1$ and $H_0$ of an oriented hypergraph, this makes homology of CRNs functorial.

\begin{definition}\label{defn:quiver-rep-category}
Let \(\mathbf{A}_2\) be the category with two objects \(0,1\) and exactly
three morphisms \(\mathrm{id}_0,\mathrm{id}_1,u\) with \(u:0\to 1\) the unique non-identity.

\noindent The functor category \([\mathbf{A}_2,\mathbf{Vect}]\) (the arrow category of \(\mathbf{Vect}\)) is defined as follows.
\begin{itemize}
  \item \textbf{Objects} are functors \(A:\mathbf{A}_2\to \mathbf{Vect}\), also called \define{quiver representations}, which are equivalently linear maps
  \[
    A(u):A(0)\longrightarrow A(1).
  \]
  To simplify, the functor $A$ is simply \(A:V\to W\) with \(V:=A(0)\) and \(W:=A(1)\).
  \item \textbf{Morphisms} \(\phi:A\Rightarrow A'\) are natural transformations. This is a pair of linear maps \(\phi_1:V\to V'\) and \(\phi_0:W\to W'\) making the square below commute,
  \[
  \begin{tikzcd}[column sep=large]
    V \arrow[r,"A"] \arrow[d,"\phi_1"'] &
    W \arrow[d,"\phi_0"] \\
    V' \arrow[r,"A'"'] &
    W'
  \end{tikzcd}
  \]
  i.e., $A'\circ \phi_1 = \phi_0 \circ A$.
\end{itemize}
\end{definition}

\begin{proposition}\label{prop:quiver}
    Stoichiometry specifies quiver representations in a functorial way, i.e.,
    \[
        \mathrm{Stoich}:\mathbf{CRN}\longrightarrow [\mathbf{A}_2,\mathbf{Vect}] \qquad \text{is a functor.}
    \]
\end{proposition}
\begin{proof}
    On objects
    \[
    \mathrm{Stoich}(\Gamma)\;:=\;C_1(\Gamma) := \R^E \xrightarrow{\,S_\Gamma\,}\R^V =: C_0(\Gamma)
    \]
    assigns to each CRN its associated chain groups.
    On morphisms a CRN morphism $\phi=(\phi_1,\phi_0):\Gamma\to\Gamma'$ already specifies a linear map $\phi_0: \mathbb{R}^V\to \mathbb{R}^{V'}$.
    We'll set $(\phi_0)_* := \phi_0$ to refer to the map between $0$-chains $C_0(\Gamma)\to C_0(\Gamma')$.
    On edges, $\phi_1:E_*\to E'_*$ is a map of pointed sets.
    This allows us to define a map $(\phi_1)_*(e_a)=e_{\phi_1(a)}$ between bases with the understanding that $e_*$ is to be identified with the zero vector. Linear extension proves that we have a linear map $(\phi_1)_*:C_1(\Gamma)\to C_1(\Gamma')$.

    Since the commutativity conditions for a CRN morphism $\phi=(\phi_0,\phi_1)$ provides $\phi_0\circ s = s' \circ \phi_1$ and $\phi_0\circ t = t' \circ \phi_1$.
    After replacing $\phi_0$ and $\phi_1$ with their linearized versions $(\phi_0)_*$ and $(\phi_1)_*$ the commutativity conditions become
    \[
        (\phi_0)_*\circ s = s' \circ (\phi_1)_* \quad \text{and} \quad (\phi_0)_*\circ t = t' \circ (\phi_1)_*
    \]
    imply
    \[
        (\phi_0)_*\circ (t-s) = (t'-s') \circ (\phi_1)_* \Leftrightarrow (\phi_0)_*\circ S_{\Gamma} = S_{\Gamma'} \circ (\phi_1)_*,
    \]
    the natural transformation condition that specifies morphisms in $[\mathbf{A}_2,\mathbf{Vect}]$.
\end{proof}

\begin{proposition}\label{prop:homology}
    Homology is a functor from the category of $A_2$ quiver representations to the category of $\{0,1\}$ graded vector spaces. Specifically, 
    \[
        H_*:[\mathbf{A}_2,\mathbf{Vect}] \to \mathbf{2grVect} \quad \text{where} \quad H_0(A):=\mathrm{coker}\, A \quad \text{and}\quad  H_1(A):=\ker A.
    \]
\end{proposition}
\begin{proof}
    The standard definition of homology as kernel modulo the image, specializes to the kernel and cokernel, respectively for a two-term chain complex.
    As these constructions are functorial, this completes the proof.
\end{proof}

\begin{corollary}
    Homology is functorial on CRNs.
\end{corollary}
\begin{proof}
    Combining Propositions~\ref{prop:quiver} and \ref{prop:homology} proves the result.
\end{proof}

\chapter{A Categorical Perspective on Topological Reduction}

Homology (and cohomology) are intricately related to the steady state dynamics of a chemical reaction network.
Specifically, we can regard elements $x\in C_0(\Gamma)=\R^V$ as the vector of concentrations $x(v_i)$ of each chemical species $v_i$.
Similarly, one can interpret elements $r\in C_1(\Gamma)=\R^E$ as giving reaction rates $r(e_A)$ for each reaction edge $e_A\in E$.
Dynamics are determined by the fundamental equation
\begin{equation}
    \frac{d}{dt} x(v_i) = Sr = \sum_{A\in E} S_{iA}r_A.
\end{equation}
Setting the left hand side to be zero is equivalent to asking for a reaction rate vector $r\in \ker S_{\Gamma}=H_1(\Gamma)\cong H^1(\Gamma)$.
Elements of $H_0(\Gamma)$ are related to so-called \emph{conserved charges} or \emph{conserved moiety} of a CRN.

Reduction methods for CRNs are intended to replace more complicated CRNs with simpler ones, while not changing their coarse, qualitative behaviors.
Prior reduction methods have tended to focus on the dynamics \cite{FEINBERG198059, FEINBERG1974775, FEINBERG19891819} or the spectrum \cite{jost2018hypergraph, veerman2022chemicalreactionnetworkslaplacian} of the CRN, but as the preceding paragraph demonstrates, both are manifestations of algebraic topological features of CRNs.
One of the first results of this section will be formulation of CRN reductions internal to the category $\mathbf{CRN}$.
This differs from \cite{Hirono2021}, where the concept of reduction is introduced less formally and relies on an orthogonal decomposition of a CRN into complementary networks.
The second major result of this section will be a categorical description of this Schur complement operation.

\begin{definition}[Reduction Morphism]
    Let $\gamma$ be a subnetwork of a CRN $\Gamma = (V,E,s,t)$. A \define{reduction morphism} associated to $\gamma$ is a CRN morphism $\phi:\Gamma \to \Gamma'$
    where
    \begin{itemize}
        \item Every reaction $e \in E_\gamma$, its image is degenerate in stoichiometry. Equivalently,
        \[
        \phi_0(s(e))=\phi_0(t(e)).
        \]
        \item $\phi_1$ is bijective on $E \setminus E_\gamma$
        \item For every reaction $e \in E \setminus E_\gamma$
        \[
        \phi_0(s(e)) = s(e) \qquad \textrm{and }\qquad \phi_0(t(e))=t(e)
        \]
    \end{itemize}
\end{definition}
\begin{example}[cf. Example 5 of \cite{Hirono2021}]
    In \cite{Hirono2021}, they consider the following CRN  $\Gamma$ where $(V, E) = (\{v_1, v_2\}, \{e_1, e_2, e_3\})$ and $s(e_2)=\delta_{v_1}$ and $t(e_2)=\delta_{v_2}$ are the Dirac deltas on $v_1$ and $v_2$, respectively.
    The subnetwork $\gamma = (V_{\gamma},E_{\gamma})=(\{v_1\},\{e_2\})$ is to be reduced.
    Graphically, the authors illustrate the reduction morphism via the following diagram:

    
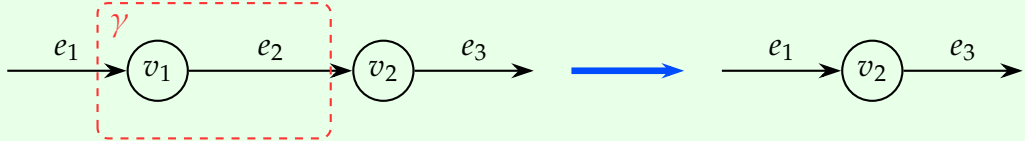
\begin{tikzpicture}[
    >={Stealth[length=3mm, width=2mm]},
    vertex/.style={circle, draw, thick, minimum size=8mm, inner sep=0pt, font=\large},
    every edge/.style={draw, thick, ->},
    label/.style={above, font=\large}]

    \coordinate (in_e1) at (0,0);
    \node[vertex] (v1) at (2,0) {$v_1$};
    \node[vertex] (v2) at (5,0) {$v_2$};
    \coordinate (out_e3) at (7,0);

    \draw[thick, ->] (in_e1) -- (v1) node[midway, label] {$e_1$};
    \draw[thick, ->] (v1) -- (v2) node[midway, label] {$e_2$};
    \draw[thick, ->] (v2) -- (out_e3) node[midway, label] {$e_3$};

    \draw[draw=red!80, dashed, thick, rounded corners] (1.2, -0.9) rectangle (4.3, 0.9);
    \node[red!80, font=\large] at (1.5, 0.6) {$\gamma$};

    \draw[->, line width=2.5pt, blue!70!cyan] (7.5,0) -- (9,0);

    \coordinate (in_e1_prime) at (9.5,0);
    \node[vertex] (v2_prime) at (11.5,0) {$v_2$};
    \coordinate (out_e3_prime) at (13.5,0);

    \draw[thick, ->] (in_e1_prime) -- (v2_prime) node[midway, label] {$e_1$};
    \draw[thick, ->] (v2_prime) -- (out_e3_prime) node[midway, label] {$e_3$};

    \end{tikzpicture}

    \noindent However, this is not a legitimate CRN morphism, as edges must be sent to edges and it is not clear where $e_2\in E$ is sent.
    Moreover, since \cite{Hirono2021} defines CRN morphisms as maps $\R^V \to \R^{V'}$ it is tempting to simply apply the projection map that forgets $v_1$, but this will produce the wrong answer as well.

    \noindent The remedy is to properly identify the quotient CRN as $\Gamma'$ where $(V',E')=(\{v_2'\},\{e_1',e_3'\})$ and $\phi_0$ is the map induced from $v_1\mapsto v_2'$ and $v_2\mapsto v_2'$.
    The pointed structure on $E_*'$ to provide a destination for $e_2$, which can then be disregarded later.
    The modified and correct picture for a CRN reduction morphism is then:

        
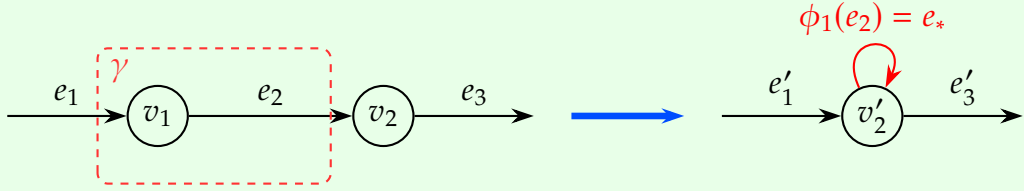
\begin{tikzpicture}[
    >={Stealth[length=3mm, width=2mm]},
    vertex/.style={circle, draw, thick, minimum size=8mm, inner sep=0pt, font=\large},
    every edge/.style={draw, thick, ->},
    label/.style={above, font=\large}
]

\coordinate (in_e1) at (0,0);
\node[vertex] (v1) at (2,0) {$v_1$};
\node[vertex] (v2) at (5,0) {$v_2$};
\coordinate (out_e3) at (7,0);

\draw[thick, ->] (in_e1) -- (v1) node[midway, label] {$e_1$};
\draw[thick, ->] (v1) -- (v2) node[midway, label] {$e_2$};
\draw[thick, ->] (v2) -- (out_e3) node[midway, label] {$e_3$};

\draw[draw=red!80, dashed, thick, rounded corners] (1.2, -0.9) rectangle (4.3, 0.9);
\node[red!80, font=\large] at (1.5, 0.6) {$\gamma$};

\draw[->, line width=2.5pt, blue!70!cyan] (7.5,0) -- (9,0);

\coordinate (in_e1_prime) at (9.5,0);
\node[vertex] (v2_prime) at (11.5,0) {$v_2'$};
\coordinate (out_e3_prime) at (13.5,0);

\draw[thick, ->] (in_e1_prime) -- (v2_prime) node[midway, label] {$e_1'$};
\draw[thick, ->] (v2_prime) -- (out_e3_prime) node[midway, label] {$e_3'$};

\draw[thick, ->, draw=red] (v2_prime) to[out=115, in=65, looseness=6] node[midway, above, font=\large, text=red] {$\phi_1(e_2)=e_*$} (v2_prime);

\end{tikzpicture}

\end{example}

Although \cite{Hirono2021} introduces several examples of CRN reduction morphisms, the example above shows that these reductions are not actually morphisms. One of the contributions of this paper is to clear up this issue by introducing a pointed definition of a CRN morphism and providing a precise definition of a reduction morphism.
This is troubling as without a pointed construction many of the ``quotient networks'' $\Gamma/\gamma$ do not exist as intended.
The rationale of \cite{Hirono2021} proceeds by considering the steady state reaction and concentration rates $(\mathbf{r},\mathbf{x})$, whose components we'd like to separate into steady state reaction and concentration rates of a sub and quotient CRN $\gamma$ and $\Gamma'\cong \Gamma / \gamma$, i.e.,
\begin{equation*}
    \mathbf{r}=
    \begin{pmatrix}
        \mathbf{r_1} \\
        \mathbf{r_2}
    \end{pmatrix}
    ,
    \mathbf{x}=
    \begin{pmatrix}
        \mathbf{x_1} \\
        \mathbf{x_2}
    \end{pmatrix}    
\end{equation*}
where $(\mathbf{r_2},\mathbf{x_2})$ is a steady state solution for $\Gamma'\cong \Gamma / \gamma$ with stoichiometry matrix $S'$.

However, as \cite{Hirono2021} observes, not every subnetwork $\gamma \subseteq \Gamma$ allows such a direct sum decomposition.
To that end they prove the following theorem.
\begin{theorem}[cf. Theorem 4 of \cite{Hirono2021}]
\label{thm:reduction}
Let \(\gamma\) be a subnetwork of a CRN. If the \define{buffering index}
\[
\lambda(\gamma)
:= -|V_\gamma| + |E_\gamma| - |(\ker S)_{\mathrm{supp}\,\gamma}| + |P_{\gamma}(\operatorname{coker} S)|
\]
satisfies \(\lambda(\gamma)=0\), and if \(s(e)(v)=0\) for every \(e \in E\setminus E_\gamma\) and every \(v \in V_\gamma\), then the CRN can be reduced. Moreover, there are isomorphisms
\[
\ker S / \ker S_\gamma \cong \ker S',
\qquad
\operatorname{coker} S / \operatorname{coker} S_\gamma \cong \operatorname{coker} S',
\]
where \(S'\) is the stoichiometric matrix of the reduced network and \(P_\gamma\) is the projection onto the subnetwork \(\gamma\).
\end{theorem}

\subsection{The Categorical Schur Complement}

As Henselman notes in his thesis \cite{henselman2017matroids}, Schur complements can be interpreted as the diagrammatic complement in the category of $A_2$ quiver representations.
We now adapt those observations to prove the following proposition:
\begin{proposition}\label{prop:Schur}
    Let $\Gamma$ be a chemical reaction network and suppose its stoichiometric map in $[\mathbf{A}_2,\mathbf{Vect}]$ admits a block decomposition
    \[
    S =
    \begin{bmatrix}
     a_{11} & a_{12}\\
     a_{21} & a_{22}
    \end{bmatrix}
    :
    C_1^{\mathrm{int}} \oplus C_1^{\mathrm{ext}}
    \longrightarrow
    C_0^{\mathrm{int}} \oplus C_0^{\mathrm{ext}}.
    \]
    
    \begin{enumerate}
        \item If $a_{11}$ is invertible, define the \define{Schur complement} of $a_{11}$ to be
        \[
            \sigma := a_{22} - a_{21}a_{11}^{-1}a_{12}.
        \]
        Then $S$ is isomorphic in $[\mathbf{A}_2,\mathbf{Vect}]$ to the direct sum $a_{11} \oplus \sigma$.
    
        \item More generally, suppose $a_{11}$ may be singular, but the compatibility conditions
        \[
            \ker(a_{11}) \subseteq \ker(a_{21}),
            \qquad
            \operatorname{im}(a_{12}) \subseteq \operatorname{im}(a_{11})
        \]
        hold. Let $a_{11}^+$ be the Moore--Penrose pseudoinverse of $a_{11}$ and define
        \[
            \sigma := a_{22} - a_{21}a_{11}^+a_{12}.
        \]
        Then $\sigma$ is well-defined, independent of the choice of generalized inverse on $\operatorname{im}(a_{12})$, and is the reduced stoichiometric map associated to the elimination of the internal block.
    \end{enumerate}
    In particular, under these hypotheses, the stoichiometric map of the reduced network may be taken to be $S_{\Gamma'} = \sigma$.
\end{proposition}

\begin{proof}
    We begin by rewriting $S_\Gamma$ as a block matrix acting on the "internal" and "external" parts of the chain groups. That is:
    \begin{equation*}
    S = \begin{bmatrix}
        a_{11} & a_{12} \\
        a_{21} & a_{22}
    \end{bmatrix}
    : \overbrace{C_1^{int} \oplus C_1^{ext}}^{C_1(\Gamma)} \to \overbrace{C_0^{int} \oplus C_0^{ext}}^{C_0(\Gamma)}
    \end{equation*}
    Define the inclusions and projections for the left hand side of $S$:
    \begin{align}\label{maps}
        \iota_{\textrm{ext}}^1:C_1^{ext} \to C_1 && \pi_{\textrm{ext}}^1:C_1 \to C_1 ^{\textrm{ext}} \\
        \iota_{\textrm{int}}^1:C_1^{int} \to C_1 && \pi_{\textrm{int}}^1:C_1 \to C_1 ^{\textrm{int}}
    \end{align}
    We define $\iota_{\textrm{ext}}^0, \iota_{\textrm{int}}^0, \pi_{\textrm{ext}}^0, \pi_{\textrm{int}}^0$ in the same way.
    If $a_{11}$ is invertible, then we set $\sigma:=a_{22} -a_{21}a_{11}^{-1}a_{12} :  C_1^{ext} \to  C_0^{ext}$.
    \newline If we consider the following automorphisms:
    \begin{align*}
        L = \begin{bmatrix}
            I & -a_{11}^{-1}a_{12} \\
            0 & I
            \end{bmatrix}:C_1 \to C_1
            &&
        R = \begin{bmatrix}
            I & 0 \\
            -a_{21}a_{11}^{-1} & I
            \end{bmatrix}:C_0 \to C_0   
    \end{align*}
    We notice and set
    \begin{equation*}
        M= RSL =  \begin{bmatrix}
            a_{11} & 0 \\
            0 & \sigma
            \end{bmatrix}
    \end{equation*} 
    Hence, the pairs $(L,R^{-1}):M \Rightarrow S$ and $(L^{-1},R):S \Rightarrow M$ are isomorphisms in the category $[\mathbf{A}_2, \textbf{Vect}]$. From \eqref{maps}, we see that these maps yield morphisms in $[\mathbf{A}_2, \textbf{Vect}]$:
    \begin{align*}
        v_1=(L\iota_{\textrm{int}}^1, R^{-1}\iota_{\textrm{int}}^0): a_{11}\Rightarrow S &&
        v_2=(L\iota_{\textrm{ext}}^1, R^{-1}\iota_{\textrm{ext}}^0): \sigma \Rightarrow S \\
        v_1^{-1}=(\pi_{\textrm{int}}^1L^{-1}, \pi_{\textrm{int}}^0R): S \Rightarrow a_{11} &&
        v_2^{-1}=(\pi_{\textrm{ext}}^1L^{-1}, \pi_{\textrm{ext}}^0R): S \Rightarrow \sigma
    \end{align*}
    These satisfy the property
    \begin{equation*}
        v_1v_1^{-1} + v_2v_2^{-1} = id_S
    \end{equation*}
    Thus, $S$ is isomorphic to the direct sum of $a_{11}$ and $\sigma$ in $[\mathbf{A}_2, \textbf{Vect}]$.

    \noindent In the case where $a_{11}$ is not invertible, we proceed with some additional conditions. Suppose that 
    \[
    \ker a_{11} \subseteq \ker a_{21}, \qquad \textrm{im}(a_{12})\subseteq \textrm{im}(a_{11})
    \]
    The latter condition reveals that for every $x \in C_{1}^{\textrm{ext}}$ there exists a $y \in C_1^{\textrm{int}}$ where 
    \[
    a_{11}y=a_{12}x
    \]
    The former condition implies that $a_{21}y$ is independent of $y$. To see this clearly, for a different $y'$, we see that 
    \[
    a_{11}(y-y') = 0
    \]
    implying $y - y' \in \ker (a_{11}) \subseteq \ker (a_{21})$, thus $a_{21}y = a_{21}y'$. Meaning that 
    \[
    x \longrightarrow a_{22}x - a_{21}y
    \]
    is well defined and defines a linear map.
    Since we had that $\textrm{im}(a_{12}) \subseteq \textrm{im}(a_{11})$. We obtain the following equation:
    \[
    a_{11}a_{11}^{+}a_{12} = a_{12}
    \]
    This means that $a_{12}^{+}a_{12}x$ lifts $a_{12}x$ via $a_{11}$. Hence
    \[
    \sigma (x) = a_{22}x - a_{21}a_{11}^{+}a_{12}x
    \]
    that is,
    \[
    \sigma = a_{22} - a_{21}a_{11}^{+}a_{12}
    \]
    Following the same argument as in the linear case, this formula is independent of the choice of our inverse, provided this agrees on $\textrm{im}(a_{12})$ mod $\ker (a_{11})$. Thus, our reduced stoichiometric matrix is determined by the block decomposition and the compatibility conditions.
\end{proof}
\begin{theorem}\label{thm:Cat_Schur}
    Let $\Gamma$ be a CRN and let $\gamma\subseteq \Gamma$ be a subnetwork with  $\lambda(\gamma)=0$ and $s(e)(v)=0$ for every $e \in E\setminus E_\gamma$ and $v \in V_\gamma$. Then the stoichiometric arrow $\mathrm{Stoich}(\Gamma)$ in $[\mathbf{A}_2,\mathbf{Vect}]$ decomposes as a categorical complement of the internal block, and the reduced CRN $\Gamma'$ obtained from the Hirono reduction satisfies
\[
\mathrm{Stoich}(\Gamma') \cong \sigma,
\]
where $\sigma$ is the Schur complement defined in Proposition \ref{prop:Schur}.
\end{theorem}

\noindent To identify the CRN associated to our Schur complement, we need to define a functor that will allow us to reconstruct a CRN when given a map between vector spaces.

\begin{definition}
    Let $[\mathbf{A}_2,\mathbf{Vect}]_R$ be the category with the same objects as $[\mathbf{A}_2,\mathbf{Vect}]$
    (linear maps $A:V\to W$), and whose morphisms are the natural transformations
    $(\alpha,\beta)$ with a natural transformation $\beta A = A'\alpha$ such that $(\alpha,\beta)$ are \emph{pushforwards}
    in the chosen bases (i.e.\ $\alpha=(\phi_1)_*$, $\beta=(\phi_0)_*$ for set maps $\phi_1,\phi_0$).
    Fix once and for all ordered bases
    $\{e_{V_i}\}$ of $V$ and $\{e_{W_j}\}$ of $W$.
    We define the \textit{reconstruction functor}
    \[
        \mathrm{Recon} : [\mathbf{A}_2,\mathbf{Vect}]_R \longrightarrow \mathbf{CRN}
    \]
    on objects by
    \[
        \mathrm{Recon}(A)
        := \Gamma_A
        := \bigl(\{e_{V_i}\}, \{e_{W_j}\}, s_A:=S_{A^-}, t_A:=S_{A^+}\bigr),
    \]
    where the species are the chosen basis of $W$, reactions are the chosen basis of $V$, and we define the stoichiometric matrix $S_A$ to be the matrix of $A$ in these bases. The source/target maps $(s_A,t_A)$ are read columnwise from the negative/positive parts of $S_A$. That is, 
    \[
    S_A:=t_A - s_A
    \]
    
    \noindent A morphism in $[\mathbf{A}_2,\mathbf{Vect}]_R$ is written
    \[
    (\alpha,\beta):(V \xrightarrow{A} W) \Longrightarrow (V' \xrightarrow{A'} W')
    \quad\text{with}\quad \alpha=(\phi_1)_*,\ \beta=(\phi_0)_*,
    \]
    i.e.\ a commuting square
    \[
    \begin{tikzcd}[column sep=large]
        V \arrow[r,"A"] \arrow[d,"\alpha"'] &
        W \arrow[d,"\beta"] \\
        V' \arrow[r,"A'"'] &
        W'
    \end{tikzcd}
    \quad\text{with}\quad
    \beta \circ A = A' \circ \alpha.
    \]
    We define
    \[
        \mathrm{Recon}(\alpha,\beta)
        := (\phi_1,\phi_0) \colon \Gamma_A \longrightarrow \Gamma_{A'}.
    \]
    The commutativity condition
    $\beta \circ A = A' \circ \alpha$ is then precisely the CRN
    compatibility condition:
    \begin{equation*}
        \begin{tikzcd}
    	E & {\mathbb{R}^V} & E & {\mathbb{R}^V} \\
    	{E'} & {\mathbb{R}^{V'}} & {E'} & {\mathbb{R}^{V'}}
    	\arrow["s_A", from=1-1, to=1-2]
    	\arrow["{\phi_1}", from=1-1, to=2-1]
    	\arrow["{\phi_0}", from=1-2, to=2-2]
    	\arrow["t_A", from=1-3, to=1-4]
    	\arrow["{\phi_1}", from=1-3, to=2-3]
    	\arrow["{\phi_0}", from=1-4, to=2-4]
    	\arrow["{s_{A'}}", from=2-1, to=2-2]
    	\arrow["{t_{A'}}", from=2-3, to=2-4]
        \end{tikzcd}
    \end{equation*}
    The commutativity $\beta A = A' \alpha$ is exactly the CRN compatibility
    $(\phi_0)_*\,S_A = S_{A'}\,(\phi_1)_*$, so this is a well-defined CRN morphism.
        
\end{definition}

\begin{proposition}
    We have an adjunction $\textrm{Stoich} \dashv \textrm{Recon}$. Meaning that for every $\Gamma \in \mathbf{CRN}$ and $(A:V\to W) \in [\mathbf{A}_2, \mathbf{Vect}]_R$ there is a natural bijection
    \[
     \mathrm{Hom}_{[\mathbf{A}_2,\mathbf{Vect}]_R}\!\big(\mathrm{Stoich}(\Gamma),A\big)\ \cong\
    \mathrm{Hom}_{\mathbf{CRN}}\!\big(\Gamma,\mathrm{Recon}(A)\big)
    \]
    sending a commuting square $(C_1(\Gamma)\!\xrightarrow{S_\Gamma}\!C_0(\Gamma))\Rightarrow (V\!\xrightarrow{A}\!W)$
    to the corresponding CRN morphism $\Gamma\to\Gamma_A$
\end{proposition}
\begin{proof}
    Fix $\Gamma = (X,E,s,t)$ and $A$ in their respective categories.
    We demonstrate a natural bijection by showing a map $\Phi:\mathrm{Hom}_{[\mathbf{A}_2,\mathbf{Vect}]_R}\!\big(\mathrm{Stoich}(\Gamma),A\big) \to \mathrm{Hom}_{\mathbf{CRN}}\!\big(\Gamma,\mathrm{Recon}(A)\big) $ 
    with an inverse $\Psi$.
    \newline Observe that in $[\mathbf{A}_2, \mathbf{Vect}]_R$, a map between our objects $\textrm{Stoich}(\Gamma)$ and $A$ is written as
    \[
    (\alpha,\beta):(C_1 \xrightarrow{S_\Gamma} C_0) \to (V \xrightarrow{A} W). \qquad  \beta \circ S_\Gamma = A \circ \alpha
    \]
    The pushforward restriction requirement in $[\mathbf{A}_2, \mathbf{Vect}]_R$ demands that we have unique set maps $\phi_1:E \to B_V$ and $\phi_0:X \to B_W$ such that $\alpha = (\phi_1)_*$ and $\beta = (\phi_0)_*$ as pushforwards from the chosen bases. Writing $S_A$ as the matrix of $A$, the commutativity becomes
    \[
    (\phi_0)_*\,S_\Gamma \;=\; S_A\,(\phi_1)_*
    \]
    which is exactly the CRN compatibility condition for a morphism
    $(\phi_1,\phi_0):\Gamma\to\Gamma_A$.
    Thus we obtain a function
    \[
    \Phi:\ \mathrm{Hom}_{[\mathbf{A}_2,\mathbf{Vect}]}\bigl(\mathrm{Stoich}(\Gamma),A\bigr)
    \longrightarrow
    \mathrm{Hom}_{\mathbf{CRN}}\bigl(\Gamma,\Gamma_A\bigr),
    \qquad
    (\alpha,\beta)\mapsto (\phi_1,\phi_0).
    \]
    Similarly, in $\mathbf{CRN}$, given a map between $\Gamma$ and $\Gamma_A$:
    \[
    (\phi_1,\phi_0): \Gamma \to \Gamma_A
    \]
    With compatability conditions for the commutative squares
    \begin{equation*}
        \begin{tikzcd}
    	E & {\mathbb{R}^X} & E & {\mathbb{R}^X} \\
    	{E_A} & {\mathbb{R}^{V}} & {E_A} & {\mathbb{R}^{V}}
    	\arrow["s_\Gamma", from=1-1, to=1-2]
    	\arrow["{\phi_1}", from=1-1, to=2-1]
    	\arrow["{\phi_0}", from=1-2, to=2-2]
    	\arrow["t_\Gamma", from=1-3, to=1-4]
    	\arrow["{\phi_1}", from=1-3, to=2-3]
    	\arrow["{\phi_0}", from=1-4, to=2-4]
    	\arrow["{s_{A}}", from=2-1, to=2-2]
    	\arrow["{t_{A}}", from=2-3, to=2-4]
        \end{tikzcd}
    \end{equation*}
    Setting $\alpha = (\phi_1)_*$ and $\beta=(\phi_0)_*$, we have that
    $\beta \circ S_\Gamma = S_A \circ \alpha$, which is exactly the commuting square condition in $[\mathbf{A}_2,\mathbf{Vect}]_R$. Meaning $(\phi_1, \phi_0)$ is a morphism in $[\mathbf{A}_2, \mathbf{Vect}]$. This defines a function
    \[
    \Psi:\ \mathrm{Hom}_{\mathbf{CRN}}\bigl(\Gamma,\Gamma_A\bigr)
    \longrightarrow
    \mathrm{Hom}_{[\mathbf{A}_2,\mathbf{Vect}]}\bigl(\mathrm{Stoich}(\Gamma),A\bigr),
    \qquad
    (\phi_1,\phi_0)\mapsto \bigl((\phi_1)_*,(\phi_0)_*\bigr).
    \]
    By construction, $\Phi$ and $\Psi$ are inverse to each other. Given a pair $(\alpha,\beta)$ and reading off $(\phi_1,\phi_0)$ recovers $(\alpha,\beta)=\big((\phi_1)_*,(\phi_0)_*\big)$, and vice versa. Naturality in $\Gamma$ and $A$ follows from functoriality of pushforwards and composition of set maps
\end{proof}

\begin{example}[cf. Fig. 2 from \cite{Hirono2021}]
\label{ex:2}
Consider the CRN \(\Gamma\) on species \(V=\{v_1,v_2,v_3,v_4\}\) and reactions
\(E=\{e_1,\dots,e_7\}\), shown in the diagram below.  Let
\[
\gamma=(V_\gamma,E_\gamma),\qquad
V_\gamma=\{v_1,v_2\},\qquad E_\gamma=\{e_1,e_2,e_3\}.
\]
Every reaction whose source uses a species of \(V_\gamma\) lies in \(E_\gamma\), so
\(\gamma\) is output-complete.  A direct computation shows that
\(\lambda(\gamma)=0\), hence \(\gamma\) is a buffering structure and Hirono reduction applies.

Applying the stoichiometric functor
\[
\mathrm{Stoich}:\mathbf{CRN}\to [A_2,\mathbf{Vect}]
\]
sends \(\Gamma\) to the arrow
\[
C_1(\Gamma)\xrightarrow{S} C_0(\Gamma).
\]
The choice of \(\gamma\) determines splittings
\[
C_1(\Gamma)=C_1^{\mathrm{int}}\oplus C_1^{\mathrm{ext}},
\qquad
C_0(\Gamma)=C_0^{\mathrm{int}}\oplus C_0^{\mathrm{ext}},
\]
where
\[
C_1^{\mathrm{int}}=\mathrm{span}\{e_1,e_2,e_3\},\quad
C_1^{\mathrm{ext}}=\mathrm{span}\{e_4,e_5,e_6,e_7\},
\]
\[
C_0^{\mathrm{int}}=\mathrm{span}\{v_1,v_2\},\quad
C_0^{\mathrm{ext}}=\mathrm{span}\{v_3,v_4\}.
\]
With respect to this decomposition,
\[
S=
\begin{bmatrix}
a_{11} & a_{12}\\
a_{21} & a_{22}
\end{bmatrix}
=
\begin{tikzpicture}[baseline=(m.center)]
  \matrix (m) [matrix of math nodes,
               left delimiter={[},
               right delimiter={]},
               row sep=2pt,
               column sep=10pt,
               nodes={anchor=center}] {
    -1 & 0 & 1 & 0 & 1 & 1 & 0 \\
     0 & 1 & -1 & 0 & 0 & 0 & 0 \\
     1 & -1 & 0 & -1 & 0 & 0 & 0 \\
     0 & 0 & 0 & 1 & -1 & 0 & -1 \\
  };

  \node[left=4pt of m-1-1] {\color{mydarkred}$v_1$};
  \node[left=8pt of m-2-1] {\color{mydarkred}$v_2$};
  \node[left=8pt of m-3-1] {$v_3$};
  \node[left=8pt of m-4-1] {$v_4$};

  \node[above=4pt of m-1-1] {\color{mydarkred}$e_1$};
  \node[above=4pt of m-1-2] {\color{mydarkred}$e_2$};
  \node[above=4pt of m-1-3] {\color{mydarkred}$e_3$};
  \node[above=4pt of m-1-4] {$e_4$};
  \node[above=4pt of m-1-5] {$e_5$};
  \node[above=4pt of m-1-6] {$e_6$};
  \node[above=4pt of m-1-7] {$e_7$};

  \draw[mydarkred,dashed,rounded corners,thick]
    ($(m-1-1.north west)+(-0.2,0.2)$)
    rectangle
    ($(m-2-3.south east)+(0.2,-0.2)$);
\end{tikzpicture}.
\]
Proposition~\ref{prop:Schur} therefore gives a well-defined generalized Schur complement
\[
\sigma
=
a_{22}-a_{21}a_{11}^+a_{12}
=
\begin{bmatrix}
-1 & 1 & 1 & 0\\
1 & -1 & 0 & -1
\end{bmatrix}.
\]
By Theorem~\ref{thm:Cat_Schur}, \(\sigma\) is the stoichiometric arrow of the reduced CRN
\(\Gamma'\) on the surviving species \(\{v_3,v_4\}\). That is,
\[
\mathrm{Stoich}(\Gamma')\cong \sigma.
\]
The reduction leaves \(e_4\) and \(e_7\) unchanged, while the altered columns of \(\sigma\)
correspond to the rewiring
\[
e_5: \left( v_4\to v_1 \right) \mapsto \left( v_4\to v_3 \right),
\qquad
e_6: \left( \varnothing\to v_1 \right) \mapsto \left( \varnothing\to v_3 \right).
\]
Once the buffering
subnetwork \(\gamma\) is chosen, the rewired CRN is determined functorially by the
categorical complement of the internal block of \(\mathrm{Stoich}(\Gamma)\).

\begin{center}
    \begin{tikzpicture}
 
 
\node[species] (v1) at (0,    0) {$v_1$};
\node[species] (v3) at (1.5,  0) {$v_3$};
\node[species] (v4) at (3,    0) {$v_4$};
\node[species] (v2) at (0.75, 2) {$v_2$};
 
\node (d1) at (-1.25, 0) {};
\node (d2) at ( 4.25, 0) {};
 
\draw[-latex,line width=0.5mm] (d1)  edge node[above] {$e_6$}        (v1);
\draw[-latex,line width=0.5mm] (v1)  edge node[above] {$e_1$}        (v3);
\draw[-latex,line width=0.5mm] (v3)  edge node[right] {$e_2$}        (v2);
\draw[-latex,line width=0.5mm] (v2)  edge node[left]  {$e_3$}        (v1);
\draw[-latex,line width=0.5mm] (v3)  edge node[above] {$e_4$}        (v4);
\draw[-latex,line width=0.5mm] (v4)  edge node[above] {$e_7$}        (d2);
\draw[-latex,line width=0.5mm,out=270,in=270] (v4) edge node[below] {$e_5$} (v1);
 
\draw[mydarkred, dashed, line width=1pt]
  (-0.45,-0.45) -- (1.1,-0.45) -- (1.1,0.45) -- (2.25,0.45)
  -- (0.85,2.45) -- (-0.45,2.45) -- (-0.45,-0.45);
 
\node[mydarkred] at (1.6, 2.1)  {\scalebox{1.1}{$\gamma$}};
 
\draw[-latex, line width=2.5pt, blue!70!cyan] (5.25,0) -- (6.45,0);
 
\node[species] (rv3) at (8.5,  0) {$v_3'$};
\node[species] (rv4) at (10.0,  0) {$v_4'$};
 
\node (rd1) at (7.25, 0) {};
\node (rd2) at (11.25,0) {};
 
\draw[-latex,line width=0.5mm] (rd1) edge node[above] {$e_6'$} (rv3);
\draw[-latex,line width=0.5mm,out=60,in=120]  (rv3) edge node[above] {$e_4'$} (rv4);
\draw[-latex,line width=0.5mm,out=240,in=300] (rv4) edge node[below] {$e_5'$} (rv3);
\draw[-latex,line width=0.5mm] (rv4) edge node[above] {$e_7'$} (rd2);
 
\end{tikzpicture}
\end{center}

\end{example}

\newpage
\printbibliography

\end{document}